\documentclass{amsart}
\usepackage[hmargin = 1 in, vmargin = 1 in]{geometry}
\usepackage{amssymb, amsmath,amsthm}
\newtheorem{theorem}{Theorem}
\newtheorem{lemma}{Lemma}

\theoremstyle{definition}
\newtheorem{mydef}{Definition}

\begin{document}
\title[On a Problem of Saffari]{Generalizations on a Problem of Saffari}
\normalsize
\author[A. P. Mangerel]{Alexander P. Mangerel}
\address{Department of Mathematics\\ University of Toronto\\
Toronto, Ontario, Canada}
\email{sacha.mangerel@mail.utoronto.ca}
\begin{abstract}
We provide a generalization of a problem first considered by Saffari and fully solved by Saffari, Erd\H{o}s and Vaughan on direct factor pairs, to arbitrary finite families of direct factors, and solve it using a method of Daboussi.  We end with a few related open problems.\end{abstract}
\maketitle
\section{Introduction}
A common problem in analytic and combinatorial number theory is to determine statistical information on the sizes of sets of products of integers from a given sequence.  For instance, the Davenport-Erd\H{o}s theorem states that given any sequence $\mathcal{A} \subseteq \mathbb{N}$, its set of multiples $\mathcal{M}(\mathcal{A}) := \{ma : a \in \mathcal{A}, m \in \mathcal{N}\}$ has logarithmic density, i.e., for $\mathcal{C} := \mathcal{M}(\mathcal{A})$, the limit
\begin{equation*}
\delta(\mathcal{C}) := \lim_{x \rightarrow \infty} \frac{1}{\log x} \sum_{n \leq x \atop n \in \mathcal{C}} \frac{1}{n}
\end{equation*}
exists (see Chapter 5 of \cite{HaR}).  Saffari \cite{Saf1} considered an inverse problem in which the set of products was found to dictate statistical information regarding the sequences that formed these products, including a particular case in which the sequences were well-behaved in the following sense:  
\begin{mydef}
Let $\mathcal{A},\mathcal{B} \subseteq \mathbb{N}$ such that $1 \in \mathcal{A} \cap \mathcal{B}$.  Then $\mathcal{A}$ and $\mathcal{B}$ are said to be \emph{direct factors} of $\mathbb{N}$ if for each $n \in \mathbb{N}$ there exists a unique pair $(a,b) \in \mathcal{A} \times \mathcal{B}$ such that $n = ab$.
\end{mydef}
Recall that a sequence $\mathcal{S}$ is said to have \textit{natural density} if the limit $\lim_{x \rightarrow \infty} x^{-1}\sum_{n \leq x \atop n \in \mathcal{S}} 1$ exists. This limit is called the \textit{(natural) density} of $\mathcal{S}$ and is denoted by $d\mathcal{S}$. In his 1976 paper, Saffari proved the following theorem:
\begin{theorem}
Let $\mathcal{A},\mathcal{B}\subseteq \mathbb{N}$ be a pair of direct factors of $\mathbb{N}$.  Then if $H(\mathcal{S}):= \sum_{s \in \mathcal{S}} \frac{1}{s}$ denotes the harmonic sum over a set $\mathcal{S}$ and  $H(A) < \infty$, then $\mathcal{A}$ and $\mathcal{B}$ have natural density.  In particular, $d\mathcal{A} = 1/H(\mathcal{B}) = 0$ and $d\mathcal{B} = 1/H(\mathcal{A})$.
\end{theorem}
In 1979, Saffari, in a joint work with Erd\H{o}s and Vaughan \cite{Saf2}, subsequently proved that, in the case when $H(\mathcal{A}) = \infty$ as well, the natural density of $\mathcal{B}$ is also zero.  Daboussi gave a simplified proof of both of these results shortly thereafter \cite{Dab}.  \\
Motivated by this initial problem, we generalize the result in the following direction:
\begin{mydef}
Let $m \geq 2$ and let $\mathcal{A}_j \subseteq \mathbb{N}$ for $1 \leq j \leq m$.  Call $\{\mathcal{A}_1,\ldots,\mathcal{A}_m\}$ an \emph{$m$-family of direct factors} for $\mathbb{N}$ if for each $n \in \mathbb{N}$ there exists a unique $m$-tuple $(a_1,\ldots a_n) \in \mathcal{A}_1 \times \cdots \times \mathcal{A}_m$ such that $n = a_1 \cdots a_m$.
\end{mydef}
It is natural to ask whether there is a similar relationship between the densities of $\mathcal{A}_i$, should they exist, in terms of the properties of the other $n-1$ sequences. We answer this question in the affirmative:
\begin{theorem}
With the notation above, $d\mathcal{A}_i = \prod_{j=1 \atop j \neq i}^n H(\mathcal{A}_j)^{-1}$, where the right side is interpreted as zero when $H(\mathcal{A}_j) = \infty$ for some $j$.
\end{theorem}
The proof follows a similar thread of ideas as that of Daboussi, but with certain necessary modifications.  In any case, we provide supplementary elaboration where needed. \\
We can construct examples of the families described in Definition 2: \\
i) Let $m$ be any positive integer and let $\{r_1,\ldots,r_{\phi(m)}\}$ be an ordering of the $\phi(m)$ residue classes coprime to $m$. Let $\mathcal{A}_j := \{n \in \mathbb{N} : p | n \Rightarrow p \equiv r_j \text{ (mod $m$)}\})$, the set of all integers composed only of primes congruent to $r_j$ mod $m$. \\
ii) Let $K/\mathbb{Q}$ be a Galois extension and let $\mathcal{A}_d$ denote the set of integers divisible only by rational primes such that the primes lying above them have relative degree $d$, where $d | [K:\mathbb{Q}]$.  This partitions the primes and thus gives a family of direct factors indexed by the divisors of the degree of the field extension. \\ \\
In the remainder of the paper, we denote by $P^+(n)$ and $P^-(n)$ the largest and smallest prime factors, respectively, of a positive integer $n$. 
\section{Proof of Theorem 2}
\begin{proof}
First, fix $y \geq 2$.  For each $n \in \mathbb{N}$ set $n_y := \prod_{p^{\nu}||n \atop p \leq y} p^{\nu}$ and let $\mathcal{A}_{i,y} := \{n : n_y \in \mathcal{A}_i\}$.  Also, for each $i$ let $\pi_i(n) = a_i \in \mathcal{A}_i$ such that $a_i$ is the $i$th component of the $n$-tuple into which $n$ decomposes (this being well-defined by hypothesis).  We remark that $P^+(ab) \leq y$ if and only if $P^+(a),P^+(b) \leq y$, and hence
\begin{equation*}
\prod_{p \leq y} (1-p^{-1})^{-1} = \sum_{P^+(n) \leq y} \frac{1}{n} = \sum_{P^+(a_1\cdots a_m) \leq y \atop a_i \in \mathcal{A}_i} \frac{1}{a_1\cdots a_m} = \prod_{i = 1}^m \left(\sum_{P^+(a_i) \leq y \atop a_i \in \mathcal{A}_i} \frac{1}{a_i} \right),
\end{equation*}
whence for each $i$, we have (provided each $\mathcal{A}_j$ is nonempty and $y$ is chosen large enough to produce a non-empty sum)
\begin{equation*}
\sum_{P^+(a_i) \leq y} \frac{1}{a_i} = \prod_{p \leq y} (1-p^{-1}) \prod_{j = 1 \atop j \neq i}^m \left(\sum_{P^+(a_j) \leq y \atop a_j \in \mathcal{A}_j} \frac{1}{a_j}\right)^{-1}.
\end{equation*}
In preparation for the remainder of the proof, we prove the following
\begin{lemma} \label{LEM1}
The density $d\mathcal{A}_{i,y}$ exists, and is equal to $\prod_{j=1 \atop j \neq i}^m \left(\sum_{P^+(a) \leq y \atop a \in \mathcal{A}_j} \frac{1}{a}\right)^{-1}$.  Moreover, if $x > 0$ and $A_i(x) := \sum_{a_i \leq x \atop a_i \in \mathcal{A}_{i}} 1$ then $A_i(x) \leq A_{i,y}(x)$.
\end{lemma}
\begin{proof}[Proof of Lemma \ref{LEM1}]
This is an elaboration of the proof of Daboussi.  We have
\begin{align*}
x^{-1}\sum_{n \leq x \atop n_y \in \mathcal{A}_i} 1 &= x^{-1}\sum_{a \leq x \atop P^+(a) \leq y, a \in \mathcal{A}_i} \sum_{m \leq \frac{x}{a} \atop P^-(m) > y} 1 = \sum_{a \leq x \atop P^+(a) \leq y, a \in \mathcal{A}_i} \frac{1}{a} \cdot \left(\frac{a}{x}\sum_{m \leq \frac{x}{a} \atop P^-(m) > y} 1\right).
\end{align*}
Note that $d\{n : P^-(n) > y\} = \prod_{p \leq y} (1-p^{-1})$ by the inclusion-exclusion principle, so the inner sum, normalized by $\frac{a}{x}$, is convergent, while the outer sum also converges (indeed it is increasing and bounded by the product $\prod_{p \leq y} (1-p^{-1})^{-1}$ for fixed $y$). Applying a discrete version of the dominated convergence theorem (say, defined by the sequence of functions $\{g_x(t)\}_x$ with $g_x(t) := f(\frac{x}{t})1_{(1,x)}(t)$) with Stieltjes integrals
\begin{equation*}
\int_1^{x} g_x(t) d\{\sum_{a \leq t} \frac{1}{a}\}
\end{equation*}
we arrive at the existence of the limit
\begin{equation*}
d\mathcal{A}_{i,y} = \lim_{x \rightarrow \infty} x^{-1}\sum_{n \leq x \atop n_y \in \mathcal{A}_{i}} 1 = \prod_{p \leq y} (1-p^{-1}) \sum_{P^+(a) \leq y \atop a \in \mathcal{A}_i} \frac{1}{a},
\end{equation*}
which shows the first part of the claim.  \\
\indent Now for each $i$, define $\phi_i: \mathcal{A}_{i} \rightarrow \mathcal{A}_{i,y}$ to be the mapping $a \mapsto \pi_i(a_y)\frac{a}{a_y}$.  Note that this is well-defined because $\frac{a}{a_y}$ has no prime factors less than $y$, and $\pi_i(a_y) \in \mathcal{A}_{i}$ with $P^+(a_y) \leq y$ by definition, so the $y$-smooth part of $\phi(a)$ is in $\mathcal{A}_{i,y}$, as required. We claim that $\phi_i$ is injective and in that case
\begin{equation*}
A_{i,y}(x) = \sum_{n \leq x \atop n_y \in \mathcal{A}_{i}} 1 \geq \sum_{n \leq x \atop n_y \in \mathcal{A}_{i}} |\phi^{-1}(a)| = \sum_{a \leq x \atop a\in \mathcal{A}_{i}} 1 = A_i(x)
\end{equation*}
which is the claim of the statement.  Indeed, if $a,a' \in \mathcal{A}_{i}$ such that $\phi(a) = \phi(a')$ then since $\frac{a_y}{\pi_i(a_y)} = \prod_{j \neq i} \pi_j(a_y)$, we have $a\prod_{j \neq i}\pi_j(a_y') = a'\prod_{j \neq i} \pi_j(a_y)$.  Since the decompositions of integers into products of elements from $\mathcal{A}_{j}$ are unique, and $a,a' \in \mathcal{A}_{i}$ while $\pi_j(a_y),\pi_j(a_y') \in \mathcal{A}_{j}$ for each $j \neq i$, it follows that $a = a'$ as $\mathcal{A}_{i}$ part, and we're done.
\end{proof}
Lemma \ref{LEM1} allows us to immediately deduce that 
\begin{equation*}
\overline{d}\mathcal{A}_{i} \leq d\mathcal{A}_{i,y} = \prod_{j=1 \atop j \neq i}^m \left(\sum_{P^+(a) \leq y \atop a \in \mathcal{A}_{j}} \frac{1}{a}\right)^{-1}.
\end{equation*}
This tells us, in particular, that if any of the sums $H(\mathcal{A}_j) = \infty$ then $d\mathcal{A}_i$ exists and is equal to zero (by taking $y \rightarrow \infty$).  \\
\indent We are left to check the case when all of $H(\mathcal{A}_j) < \infty$.  We need a lower bound to match the upper bound in the lemma to finish the proof. In this direction, we establish the following
\begin{lemma}
The following lower bound holds:
\begin{equation*}
\underline{d}\mathcal{A}_{i} \geq d\mathcal{A}_{i,y} + 1-d\mathcal{A}_{i,y}\prod_{j = 1 \atop j \neq i}^m \left(\sum_{a_j \in \mathcal{A}_j} \frac{1}{a_j} \right)
\end{equation*}
for each $1 \leq i \leq m$.
\end{lemma}
\begin{proof}
In what follows, let $1_i$ denote the characteristic function of $\mathcal{A}_{i}$ for each $i$.  Remark that $a \in \mathcal{A}_{i}$ if, and only if, the $n$-tuple representing $a$ consists of 1 at every component except for the $i$th component.  It follows that $a \in \mathcal{A}_{i}$ is representable as $a = a_1 \cdots a_n$ if, and only if, $(1_j-\delta)(n) = 0$ for each $j \neq i$, where $\delta(n)$ is 1 or 0 according to whether $n=1$ or not. \\
\indent For each $k \notin \mathcal{A}_{i}$ there exists a set $S_k \subseteq \{1,\ldots,n\}\backslash \{i\}$ such that $\pi_j(k) \neq 1$ if and only if $j \in S_k$, and by construction the converse that any such set corresponds to an element in the complement of $\mathcal{A}_{i}$ also holds.  Thus, we can form a partition of $\mathbb{N} \backslash \mathcal{A}_{i}$. Write $f_{S_k}$ to be its characteristic function. For each $S \subset \{1,\ldots,n\}\backslash \{i\}$ let $\mathcal{V}_S$ denote the set of integers $k \notin \mathcal{A}_i$ such that $S_k = S$ by the notation above.  Then $\{\mathcal{V}_S : S \subset \{1,\ldots,n\}\backslash\{i\}, |S| > 0\}$ forms a partition of $\mathbb{N}\backslash \mathcal{A}_i$, whence
\begin{align*}
x^{-1}\sum_{k \leq x} 1_i(k) &= 1-x^{-1}\sum_{S \subseteq \{1,\ldots,n\}\backslash \{i\} \atop |S| > 0} \sum_{k \in \mathcal{V}_S} f_S(k) = 1-\sum_{S \subseteq \{1,\ldots,n\}\backslash \{i\} \atop |S| > 0} \sum_{k \leq x \atop \pi_j(k) \neq 1 \leftrightarrow j \in S} \frac{1}{k}\cdot \frac{k}{x}\sum_{m \leq \frac{x}{k}} 1_i(m) \\
&\geq 1-\sum_{S \subseteq \{1,\ldots,n\}\backslash \{i\} \atop |S| > 0} \sum_{k \leq x \atop \pi_j(k) \neq 1 \leftrightarrow j \in S} \frac{1}{k}\cdot \frac{k}{x}\mathcal{A}_{i,y}\left(\frac{x}{k}\right).
\end{align*}
Appealing once again to the Dominated Convergence Theorem, we have
\begin{equation*}
\underline{d}\mathcal{A}_i \geq 1-d\mathcal{A}_{i,y}\sum_{S \subseteq \{1,\ldots,n\}\backslash \{i\} \atop |S| > 0} \sum_{k: \pi_j(k) \neq 1 \leftrightarrow j \in S} \frac{1}{k}.
\end{equation*}
As a result of the partition created, we have
\begin{equation*}
\sum_{k: \pi_j(k) \neq 1 \leftrightarrow j \in S} \frac{1}{k} = \prod_{j \in S} \left(\sum_{a_j \in \mathcal{A}_j} \frac{1}{a} - 1\right),
\end{equation*}
and so the sum above becomes, after introducing the contribution for $S = \emptyset$,
\begin{align*}
\underline{d}\mathcal{A}_i &\geq 1-d\mathcal{A}_{i,y}\sum_{S \subseteq \{1,\ldots,n\}\backslash \{i\}}\prod_{j \in S} \left(\sum_{a_j \in \mathcal{A}_j} \frac{1}{a} - 1\right) + d\mathcal{A}_{i,y}\\
&= d\mathcal{A}_{i,y} + 1-d\mathcal{A}_{i,y}\prod_{j = 1 \atop j \neq i}^n \left(1+\left(\sum_{a_j \in \mathcal{A}_j} \frac{1}{a_j} -1\right)\right) = d\mathcal{A}_{i,y}+1-d\mathcal{A}_{i,y}\prod_{j = 1 \atop j \neq i}^n \left(\sum_{a_j \in \mathcal{A}_j} \frac{1}{a_j}\right),
\end{align*}
which proves the lemma.
\end{proof}
To finish the proof, we write $\sum_{a_j \in \mathcal{A}_j} \frac{1}{a_j} = \sum_{a_j \in \mathcal{A}_{j} \atop P^+(a_j) \leq y} \frac{1}{a_j} + \sum_{a_j \in \mathcal{A}_j \atop P^+(a_j) > y} \frac{1}{a_j}$, noting that the second sum vanishes as $y \rightarrow \infty$.  We have by the first lemma that
\begin{equation*}
\underline{d}\mathcal{A}_{i} \geq d\mathcal{A}_{i,y} - d\mathcal{A}_{i,y} \left(\prod_{j = 1 \atop j \neq i}^n \left(H\left(\mathcal{A}_{j,y}\right) + H\left(\mathcal{A}_{i} \backslash \mathcal{A}_{j,y}\right)\right) - \prod_{j = 1 \atop j \neq i}^n H\left(\mathcal{A}_{j,y}\right)\right).
\end{equation*}
Remark that in the bracketed term, the product $\prod_{j = 1 \atop j \neq i}^n H(\mathcal{A}_{j,y})$ is cancelled off, and each remaining term is multiplied by some factor $H(\mathcal{A}_j \backslash \mathcal{A}_{j,y})$ for $j \neq i$.  As all $H(\mathcal{A}_j)$ are assumed finite, the former terms go to zero as $y \rightarrow \infty$, and hence, for any $\epsilon > 0$ we can choose $y$ (depending only on $\epsilon$ and $n$) large enough such that
\begin{equation*}
\underline{d}\mathcal{A}_{i} \geq d\mathcal{A}_{i,y} - \epsilon.
\end{equation*}
Thus, $d\mathcal{A}_i$ exists and is equal to $\lim_{y \rightarrow \infty} d\mathcal{A}_{i,y}$, implying the theorem.
\end{proof}
\section{Open Problems and other Generalizations}
Instead of considering collections of integer sequences representing all positive integers uniquely, we could restrict to respresentations of some subsequene of $\mathbb{N}$.
\begin{mydef}
Let $S \subseteq \mathbb{N}$.  Call $\mathcal{A}_1,\mathcal{A}_2$ a \emph{pair of direct factors for $S$} if for each $s \in S$ there exists a unique pair $(a_1,a_2) \in \mathcal{A}_1 \times \mathcal{A}_2$ such that $s = a_1a_2$.
\end{mydef}
Note that it may be that $S \subsetneq \{a_1a_2 : (a_1,a_2) \in \mathcal{A}_1 \times \mathcal{A}_2\}$.  All we require is that the map $(a,b) \mapsto ab$ be an injection on the preimage of $S$. We seek to know whether any analogous relationship will exist between $\mathcal{A}_1$ and $\mathcal{A}_2$ according to the properties of $S$ (which, for instance, might require $S$ to possess natural density).\\
\indent Another natural question is to classify the set of direct factors of $\mathbb{N}$, and more generally, of sequences $S$ of the type considered in answering the above problem.  We may remark, for example, that there is no $\mathcal{A}$ such that $(\mathcal{A},\mathcal{A})$ is a direct factor pair even if we do not distinguish between $(a,a')$ and $(a',a)$. Indeed, $\mathcal{A}$ must contain every prime factor and 1, implying that it cannot contain any squares of primes and hence must contain all cubes of primes.  In this case, however, it will not contain any fourth powers of primes since otherwise one should have $p^4 = p\cdot p^3 = p^4 \cdot 1$.  As a result, $\mathcal{A} \cdot \mathcal{A}$ cannot contain any fifth powers of primes, as the only smaller such powers are cubes and the primes themselves.\\
\indent Conversely, it is possible that a sequence have infinitely many direct factor pairs.  Indeed, let $S \subset \mathbb{N}$ be a primitive sequence, i.e., such that for any two $s',s \in S$ with $s' < s$ then $s' \nmid s$. Let $\{S_1,S_2\}$ be a partition of $S$ and set $S' := S \cup \{1\}$ and $S_j' := S_j \cup \{1\}$ for $j = 1,2$.  Then clearly each $s \in S'$ has the form $s = s\cdot 1$ for $s \in S_1'$ or $s \in S_2'$, and moreover if $s = s_1s_2$ then one of $s_1$ and $s_2$ must be 1, otherwise $s_j | s$, contradicting primitivity.  If $S$ is an infinite such sequence (these do exist, an example furnished by the set $\{p_i^i : i \geq 1 \}$, where $p_i$ denotes the $i$th prime) then there are infinitely many such partitions providing distinct direct factor pairs. 

\end{document}